\documentclass[11pt, hyperref]{article}
\usepackage{authblk}
\usepackage{amsmath,amssymb,amsfonts,amsthm,graphicx}
\usepackage[colorlinks=true,linkcolor=blue,citecolor=blue,urlcolor=blue]{hyperref}
\usepackage{subcaption,float}
\usepackage{cite}

\newtheorem{theorem}{Theorem}[section]
\newtheorem{lemma}[theorem]{Lemma}
\theoremstyle{definition}
\newtheorem{definition}[theorem]{Definition}
\newtheorem{remark}[theorem]{Remark}

\title{Port-Hamiltonian Control and Structure-Preserving Algorithm for Grid-Forming SVGs}
\author[1]{Jiaxin Qian\thanks{qianjiaxin@amss.ac.cn}}
\author[2]{Feng Ji\thanks{jameskeating@163.com}}
\author[1]{Sixu Wu\thanks{lucaswu@amss.ac.cn}}
\author[3]{Mingyang Liu\thanks{liumingyang3@ha.sgcc.com.cn}}
\author[1,4]{Yifa Tang\thanks{tyf@lsec.cc.ac.cn}\footnote{corresponding author}}
\affil[1]{State Key Laboratory of Mathematical Sciences, Academy of Mathematics and Systems Science, Chinese Academy of Sciences, Beijing 100190, China}
\affil[2]{State Key Laboratory of Advanced Transmission Technology (China Electric Power Research Institute Ltd), Changping District, Beijing 102200, China}
\affil[3]{State Grid Henan Electric Power Research Institute, Zhengzhou 450052, China}
\affil[4]{School of Mathematical Sciences, University of Chinese Academy of Sciences, Beijing 100049, China}

\date{}

\begin{document}

\maketitle

\begin{abstract}
This paper presents a port-Hamiltonian (PH) modeling, control, and structure-preserving simulation framework for grid-forming static var generators (SVGs). A PH model is established that captures energy exchange among the inductor, capacitor, and DC-link storage ports. Since external disturbances cannot be fully canceled by feedback, an input-to-state stable (ISS) controller is designed to steer subsystem states to zero while minimizing disturbance effects. The controller contains only three tunable parameters with clear physical interpretations and is robust against input errors. A Dirac-structure-preserving midpoint rule is developed, which exactly conserves the Hamiltonian energy when disturbances are absent. Numerical comparisons show that the ISS controller achieves faster settling, smaller offset, and lower control effort than a conventional PI controller, and the structure-preserving midpoint rule maintains exact energy conservation and superior long-term accuracy over standard Runge-Kutta methods.
\end{abstract}


\section{Introduction}\label{introduction}
Grid-forming static var generators (SVGs) are playing an increasingly vital role in modern power systems, especially as the penetration of renewable energy sources continues to rise and conventional synchronous generators are progressively being decommissioned~\cite{qaisar2025grid}. Unlike grid-following converters, which rely on a strong grid voltage reference, grid-forming converters actively establish the grid voltage and frequency, thereby providing essential inertia and voltage support to the system~\cite{blaabjerg2017distributed}.

Among various grid-forming devices, the SVG is particularly valued for its ability to regulate reactive power and maintain bus voltage. Traditionally, PI control has been the most widely used strategy in SVG systems due to its structural simplicity and ease of implementation~\cite{teng2024adaptive}. However, conventional PI controllers exhibit several intrinsic drawbacks that the present work directly addresses. First, their performance heavily depends on the accuracy of the mathematical model, yet they lack robustness against input errors caused by discrete-time implementation~\cite{xuadaptive}. Second, the fixed PI gains often fail to adapt to changing grid conditions, leading to deteriorated dynamic behavior under system uncertainties and external disturbances; moreover, the PI controller typically requires trial‑and‑error or LQR‑based tuning of multiple parameters without clear physical interpretation. Third, under strong disturbances, PI control produces noticeable steady‑state offsets and oscillatory responses, while demanding relatively high control input magnitudes. Adaptive PI methods have been developed to address some of these limitations, yet they still struggle with issues such as incomplete active‑reactive decoupling and limited robustness under strong disturbances~\cite{teng2024adaptive,xuadaptive}.

To overcome these shortcomings, energy-based control methods have emerged as a promising alternative~\cite{ortega2002interconnection}. Among these, the port-Hamiltonian framework offers a particularly powerful paradigm, modeling the system as a collection of interconnected energy storage, dissipation, and external ports, with each port associated with an energy exchange pathway~\cite{van2020port}. Consequently, control design can be centered around the management of system energy variation~\cite{maschke1993port,ortega2002interconnection}.

Moreover, a port-Hamiltonian system inherently possesses a geometric structure known as the Dirac structure, which encodes how energy flows among the system's components and fundamentally determines the trend of energy changes~\cite{van2020port,mehrmann2019structure}. Preserving this structure in numerical simulations is critical; otherwise, conventional discretization algorithms may artificially inject or dissipate energy, leading to unphysical numerical drift—a problem that remains largely overlooked in existing SVG research~\cite{kumar2025port,hairer2006geometric,mehrmann2019structure}.

In a broader context, the importance of preserving geometric structures in numerical integration was first recognized by Feng Kang, who pioneered symplectic algorithms for Hamiltonian systems in the 1980s~\cite{feng1984difference}. His seminal
work established the theoretical framework for constructing symplectic difference schemes, demonstrating that preserving the symplectic structure of Hamiltonian dynamics yields superior long-term stability and energy conservation~\cite{Feng2020}. Reference~\cite{tang2026geometric} offers a comprehensive account of structure-preserving algorithms and machine learning techniques for dynamical systems.

In this paper, we address these challenges by developing a comprehensive port-Hamiltonian framework for grid-forming SVGs, together with an associated energy-shaping control strategy and a Dirac-structure-preserving numerical scheme. Compared with PI-based approaches, the input‑to‑state stable (ISS) controller proposed in this paper operates with only three physically interpretable parameters, guarantees bounded energy even when disturbances cannot be canceled exactly, exhibits strong robustness against input errors, and achieves faster settling, smaller offset, and lower control effort. Compared with non-structure-preserving Runge–Kutta discretizations, the structure-preserving Runge-Kutta algorithm proposed in this paper exactly conserves energy in the absence of external disturbances, enabling reliable long-term simulations even under strong control actions.

This paper is organized as follows. Section \ref{Model} presents a port‑Hamiltonian modeling of the grid‑forming SVG system, showing how the energy function is affected only by external disturbances. Section \ref{con} designs a feedback controller based on input‑to‑state stability that drives the subsystem back to the desired equilibrium while maintaining robustness against input errors. Section \ref{algorithm} introduces a Dirac structure preserving numerical algorithm, including the definition of Dirac structures, the definition of Dirac-structure-preserving algorithm and a two-stage second-order Runge-Kutta method preserving the Dirac structure of grid-forming SVGs generated by the midpoint rule. Section \ref{numexp} provides numerical experiments that compare the proposed controller with a conventional PI controller and compare the structure‑preserving algorithm with a standard Runge‑Kutta method. Finally, Section \ref{conclusion} concludes the paper by summarizing the advantages of the proposed methods and discussing two remaining problems.

\section{A Port-Hamiltonian Modeling for Grid-Forming SVG}\label{Model}

Assuming zero resistance, a grid-forming static var generator (SVG) satisfies the following equation: \begin{equation}
	\begin{aligned}
	\frac{d}{dt}\begin{bmatrix}
		x_1\\x_2\\x_3\\x_4\\x_5
	\end{bmatrix}=\begin{bmatrix}
		0&\omega&-\frac{1}{L}&0&0\\
		-\omega&0&0&-\frac{1}{L}&0\\
		\frac{1}{C}&0&0&\omega&0\\
		0&\frac{1}{C}&-\omega&0&0\\
		0&0&0&0&0
	\end{bmatrix}\begin{bmatrix}
		x_1\\x_2\\x_3\\x_4\\x_5
	\end{bmatrix}-\frac{1}{C}\begin{bmatrix}
		0\\0\\i_g^d\\i_g^q\\0
	\end{bmatrix}+\begin{bmatrix}
		\frac{1}{L}&0\\0&\frac{1}{L}\\0&0\\0&0\\-x_1&-x_2
	\end{bmatrix}\begin{bmatrix}
		u_1\\u_2
	\end{bmatrix}
\end{aligned}
\label{SVG}
\end{equation}
where $x_1,x_2,x_3,x_4,x_5$ are state variables. Specifically, $x_1$ and $x_2$ represent the two components of the filtered inductor current, $x_3$ and $x_4$ represent the two components of the output voltage, and $x_5$ represents the energy of a DC capacitor. The parameters $i_g^d$ and $i_g^q$ are external disturbances, while $u_1$ and $u_2$ are control inputs. Finally, $L$ denotes inductance, $C$ denotes capacitance, and $\omega$ denotes angular velocity. The quantities $L,C$ and $\omega$ can be regarded as three constants in a grid-forming SVG system.

Suppose the external disturbances satisfy $i_g^d=i_g^q=0$. Then the fixed points of system \eqref{SVG} satisfy $x_1^*=x_2^*=x_3^*=x_4^*=0$, and $x_5^*$ is determined by the initial condition $\mathbf{x}(0)=(x_1(0),x_2(0),x_3(0),x_4(0),x_5(0))$. Indeed, defining $H(\mathbf{x})=\frac{L}{2}(x_1^2+x_2^2)+\frac{C}{2}(x_3^2+x_4^2)+x_5$ which is the total energy of the grid-forming SVG system, we have $\frac{dH}{dt}=0$. Now use $H$ to rewrite \eqref{SVG} in port-Hamilton form defined below~\cite{maschke1993port,van2020port}.

\begin{definition}
	A port-Hamiltonian system is a control system of the form \begin{align}\label{PH}
		\dot{\mathbf{x}}=[J(\mathbf{x})-R(\mathbf{x})]\nabla H(\mathbf{x})+B(\mathbf{x})\mathbf{d}+C(\mathbf{x})\mathbf{u}
	\end{align}
	where \(\mathbf{x} \in \mathbb{R}^n\) is the state vector, \(\mathbf{d} \in \mathbb{R}^k\) is the vector of external disturbances, and \(\mathbf{u} \in \mathbb{R}^m\) is the control input vector. Moreover, \(J(\mathbf{x}) \in \mathbb{R}^{n \times n}\) is a skew-symmetric matrix, \(R(\mathbf{x}) \in \mathbb{R}^{n \times n}\) is a positive semi-definite symmetric matrix, and \(B(\mathbf{x}) \in \mathbb{R}^{n \times k}\), \(C(\mathbf{x}) \in \mathbb{R}^{n \times m}\) are arbitrary matrices. The scalar function \(H: \mathbb{R}^n \to \mathbb{R}\) is called the Hamiltonian function or energy function.
\end{definition}

For system \eqref{SVG}, let \(\mathbf{x} = (x_1, x_2, x_3, x_4, x_5)\) be the state vector and \(H(\mathbf{x}) = \frac{L}{2}(x_1^2 + x_2^2) + \frac{C}{2}(x_3^2 + x_4^2) + x_5\) be the Hamiltonian function. Then this system can be rewritten as \begin{align}\label{PHMod}
	\dot{\mathbf{x}}=J\nabla H(\mathbf{x})+B\mathbf{d}+C(\mathbf{x})\mathbf{u}
\end{align}
where $J=\begin{bmatrix}
		0&\frac{\omega}{L}&-\frac{1}{CL}&0&0\\
		-\frac{\omega}{L}&0&0&-\frac{1}{CL}&0\\
		\frac{1}{CL}&0&0&\frac{\omega}{C}&0\\
		0&\frac{1}{CL}&-\frac{\omega}{C}&0&0\\
		0&0&0&0&0
	\end{bmatrix}$, $B=\begin{bmatrix}
	0&0\\0&0\\\frac{1}{C}&0\\0&\frac{1}{C}\\0&0
	\end{bmatrix}$, $C(\mathbf{x})=\begin{bmatrix}
	\frac{1}{L}&0\\0&\frac{1}{L}\\0&0\\0&0\\-x_1&-x_2
	\end{bmatrix}$, $\mathbf{d}=\begin{bmatrix}
	i_g^d\\i_g^q
	\end{bmatrix}$ and $\mathbf{u}=\begin{bmatrix}
	u_1\\u_2
	\end{bmatrix}$.

The energy function \(H: \mathbb{R}^5 \to \mathbb{R}\) satisfies
\[
\frac{dH}{dt} = \nabla H(\mathbf{x})^\top \dot{\mathbf{x}} = -x_3 i_g^d - x_4 i_g^q,
\]
which means that the change of \(H\) depends only on the disturbance input \(\mathbf{d} = \begin{bmatrix} i_g^d & i_g^q \end{bmatrix}^\top\).

Because the target state variables to be controlled are the filtered inductor current and the output voltage, we now consider the subsystem \begin{align}\label{subPHMod}
\dot{\mathbf{x}}_{\text{sub}}=J_\text{sub}\nabla H(\mathbf{x}_{\text{sub}})+B_\text{sub}\mathbf{d}+C_\text{sub}\mathbf{u}
\end{align}
where \(\mathbf{x}_{\text{sub}} = [x_1, x_2, x_3, x_4]^\top\) is the state vector and the total energy of capacitor and inductor \(H_0(\mathbf{x}_{\text{sub}}) = \frac{L}{2}(x_1^2 + x_2^2) + \frac{C}{2}(x_3^2 + x_4^2)\) is the Hamiltonian function. This subsystem is also a port-Hamiltonian system, where $J_{\text{sub}}=\begin{bmatrix}
		0&\frac{\omega}{L}&-\frac{1}{CL}&0\\
		-\frac{\omega}{L}&0&0&-\frac{1}{CL}\\
		\frac{1}{CL}&0&0&\frac{\omega}{C}\\
		0&\frac{1}{CL}&-\frac{\omega}{C}&0
	\end{bmatrix}$, $B_{\text{sub}}=\begin{bmatrix}
		0&0\\0&0\\\frac{1}{C}&0\\0&\frac{1}{C}
	\end{bmatrix}$, $C_{\text{sub}}=\begin{bmatrix}
		\frac{1}{L}&0\\0&\frac{1}{L}\\0&0\\0&0
	\end{bmatrix}$, $\mathbf{d}=\begin{bmatrix}
	i_g^d\\i_g^q
	\end{bmatrix}$ and $\mathbf{u}=\begin{bmatrix}
	u_1\\u_2
	\end{bmatrix}$.

Based on this port-Hamiltonian formulation, the next section designs a feedback controller for the subsystem. The control objective is to drive the states $x_1$ through $x_4$ back to zero in the presence of external disturbances, while ensuring robustness against practical imperfections such as input errors.

\section{A Feedback Controller Based on Energy Function}\label{con}
From section \ref{Model}, the system to be controlled is the subsystem \eqref{subPHMod}. The control objective is to ensure that \(\mathbf{x}_{\text{sub}}\) returns nearly to \(\mathbf{0}\) despite the disturbance input \(\mathbf{d}\). For the port-Hamiltonian system \eqref{PH}, if there exists a matrix \(K(\mathbf{x}) \in \mathbb{R}^{m \times k}\) such that \(B(\mathbf{x}) = C(\mathbf{x}) K(\mathbf{x})\), then setting \(\mathbf{u} = -K(\mathbf{x})\mathbf{d} + \mathbf{v}\) transforms \eqref{PH} into
\[
\dot{\mathbf{x}} = [J(\mathbf{x}) - R(\mathbf{x})] \nabla H(\mathbf{x}) + C(\mathbf{x}) \mathbf{v}.
\]
If there exists a controller \(\mathbf{v} = \mathbf{h}(\mathbf{x})\) such that
\[
\frac{dH}{dt} = \nabla H(\mathbf{x})^\top \dot{\mathbf{x}} \le -\alpha H(\mathbf{x}), \quad \alpha > 0,
\]
then \(H\) converges to zero.

However, in system \eqref{subPHMod}, there is no \(K(\mathbf{x}_{\text{sub}}) \in \mathbb{R}^{2 \times 2}\) such that
\[
\begin{bmatrix} \frac{1}{L} & 0 \\ 0 & \frac{1}{L}\\0&0\\0&0 \end{bmatrix} K(\mathbf{x}_{\text{sub}}) = \begin{bmatrix} 0 & 0 \\ 0 & 0 \\ \frac{1}{C} & 0 \\ 0 & \frac{1}{C} \end{bmatrix},
\]
so the disturbance input \(\begin{bmatrix} i_g^d & i_g^q \end{bmatrix}^\top\) cannot be completely canceled. Therefore, we design a feedback controller that minimizes the impact of the disturbance on \(H_0\) as much as possible. A class of controller design methods is based on input-to-state stability of port-Hamilton systems~\cite{sontag1989smooth,sontag1995characterizations,van2000l2}.

\subsection{Input-to-State Stability}

\begin{definition}
	The port-Hamiltonian system \eqref{PH} is called input-to-state stable (ISS) if \(H\) satisfies
	\[
	\frac{dH}{dt} \le -\alpha H + \beta \|\mathbf{d}\|^2,
	\]
	where \(\|\mathbf{d}\| = \sqrt{\sum_{i=1}^k d_i^2}\) for \(\mathbf{d} \in \mathbb{R}^k\), and \(\alpha, \beta > 0\) are constants.
\end{definition}

If a port-Hamiltonian system is ISS, then by the second comparison theorem for ODEs~\cite{arnold1992ordinary},
\[
H(\mathbf{x}(t)) \le H_0 e^{-\alpha t} + \beta \int_0^t \|\mathbf{d}(\tau)\|^2 e^{-\alpha(t-\tau)} d\tau \le H_0 e^{-\alpha t} + \frac{\beta}{\alpha} \sup \|\mathbf{d}\|^2,
\]
where $H_0=H(\mathbf{x}(0))$. Furthermore, if \(\frac{\beta}{\alpha}\) is small, \(H(\mathbf{x}(t))\) will ultimately be contained within a small range; if \(\alpha\) is large, \(H(\mathbf{x}(t))\) will decay rapidly.

For system \eqref{PH}, to ensure that \(H\) returns to a bounded range, we need a feedback controller \(\mathbf{u} = \mathbf{h}(\mathbf{x})\) that renders the system ISS, and the parameters \(\alpha, \beta\) determine the size of the ultimate bound and the decay rate.

Inspired by $L_2$-gain-based control methods~\cite{sontag1989smooth,sontag1995characterizations,van2000l2,van2020port}, $\mathbf{h(x)}$ could satisfy the inequality in the following system:

\begin{theorem}\label{ine}
	The port-Hamiltonian system \eqref{PH} with control input \(\mathbf{u} = \mathbf{h}(\mathbf{x})\) is input-to-state stable if \begin{align}\label{condISS}
		\nabla H^\top[C(\mathbf{x})\mathbf{h(x)}+\frac{1}{4\beta}B(\mathbf{x})B(\mathbf{x})^\top\nabla H-R(\mathbf{x})\nabla H]+\alpha H\leq 0
	\end{align} for all $\mathbf{x}\in\mathbb{R}^n$.
\end{theorem}

\begin{proof}
	Substitute \(\mathbf{u} = \mathbf{h}(\mathbf{x})\) into the derivative of \(H\):
	\[
	\begin{aligned}
		\frac{dH}{dt} &= \nabla H^\top \dot{\mathbf{x}} \\
		&= \nabla H^\top \bigl( -R\nabla H + B\mathbf{d} + C\mathbf{h} \bigr) \\
		&\le -\alpha H - \frac{1}{4\beta} \nabla H^\top B B^\top \nabla H + \nabla H^\top B \mathbf{d} \\
		&\le -\alpha H + \beta \|\mathbf{d}\|^2,
	\end{aligned}
	\]
	where we used the Cauchy–Schwartz inequality and the elementary inequality
	\[
	\nabla H^\top B \mathbf{d} \le \|B^\top \nabla H\| \|\mathbf{d}\| \le \frac{1}{4\beta} \|B^\top \nabla H\|^2 + \beta \|\mathbf{d}\|^2.
	\]
\end{proof}

\subsection{Controller Design}\label{controller}

For system \eqref{subPHMod}, inequality \eqref{condISS} becomes
\[
(x_1 h_1(\mathbf{x}_{\text{sub}}) + x_2 h_2(\mathbf{x}_{\text{sub}})) + \frac{1}{4\alpha\epsilon} (x_3^2 + x_4^2) + \alpha \left[ \frac{L}{2}(x_1^2 + x_2^2) + \frac{C}{2}(x_3^2 + x_4^2) \right] \le 0,
\]
where \(\epsilon = \frac{\beta}{\alpha}\). The parameter \(\epsilon\) determines the size of the region to which \(H_0\) eventually shrinks, and \(\alpha\) determines the decay rate.

Define the cost functional \(J(\mathbf{h}) = \int_0^T \|\mathbf{h}(\mathbf{x}_{\text{sub}}(t))\|^2 dt\). To minimize \(J(\mathbf{h})\), we minimize \(\|\mathbf{h}(\mathbf{x}_{\text{sub}})\|^2\) for each \(\mathbf{x}_{\text{sub}} \in \mathbb{R}^4\). By the Cauchy–Schwarz inequality,
\[
-\sqrt{x_1^2 + x_2^2} \|\mathbf{h}\| \le x_1 h_1 + x_2 h_2.
\]Moreover,
\[
x_1 h_1 + x_2 h_2 \le -\left[ \frac{1}{4\alpha\epsilon}(x_3^2 + x_4^2) + \alpha \left( \frac{L}{2}(x_1^2 + x_2^2) + \frac{C}{2}(x_3^2 + x_4^2) \right) \right].
\]
The right-hand side is a constant for a fixed \(\mathbf{x}_{\text{sub}}\). Set
\[
x_1 h_2(\mathbf{x}_{\text{sub}}) = x_2 h_1(\mathbf{x}_{\text{sub}}), \quad x_1 h_1 + x_2 h_2 = A,
\]
with
\[
A = -\left[ \frac{1}{4\alpha\epsilon}(x_3^2 + x_4^2) + \alpha \left( \frac{L}{2}(x_1^2 + x_2^2) + \frac{C}{2}(x_3^2 + x_4^2) \right) \right].
\]
Then the controller is given by
\[
\begin{cases}
	h_1(\mathbf{x}_{\text{sub}}) = \dfrac{A x_1}{x_1^2 + x_2^2} = -x_1 \left( \dfrac{1}{4\alpha\epsilon} \dfrac{x_3^2 + x_4^2}{x_1^2 + x_2^2} + \dfrac{\alpha L}{2} + \dfrac{\alpha C}{2} \dfrac{x_3^2 + x_4^2}{x_1^2 + x_2^2} \right), \\[1em]
	h_2(\mathbf{x}_{\text{sub}}) = \dfrac{A x_2}{x_1^2 + x_2^2} = -x_2 \left( \dfrac{1}{4\alpha\epsilon} \dfrac{x_3^2 + x_4^2}{x_1^2 + x_2^2} + \dfrac{\alpha L}{2} + \dfrac{\alpha C}{2} \dfrac{x_3^2 + x_4^2}{x_1^2 + x_2^2} \right).
\end{cases}
\]

\begin{remark}
	Because the controller is applied at discrete time instants in practical application, the control input is not always equal to $\mathbf{h}(\mathbf{x}_{\text{sub}})$. So if \(x_1 = x_2 = 0\) or \(\frac{x_3^2 + x_4^2}{x_1^2 + x_2^2}\) is large, in order to prevent system divergence caused by excessive input error, we replace \(\frac{A}{x_1^2 + x_2^2}\) by \(\frac{1}{4\alpha\epsilon} + \frac{\alpha L}{2} + \frac{\alpha C}{2}\).
\end{remark}

\subsection{Input Error Robustness}
In practical circuit control, input errors are inevitable because control is applied at discrete time instants, and \(\mathbf{h}(\mathbf{x}_{\text{sub}})\) is evaluated based on the state variables at those instants. Therefore, we verify the robustness of the controller \(\mathbf{h}(\mathbf{x}_{\text{sub}})\) against input errors.

\begin{theorem}\label{ROB}
	Let \(\mathbf{u} = \mathbf{v} + \mathbf{h}(\mathbf{x}_{\text{sub}})\), where
	\[
	h_1 = \frac{A x_1}{x_1^2 + x_2^2}, \quad h_2 = \frac{A x_2}{x_1^2 + x_2^2}.
	\]
	Then
	\[
	H_0(\mathbf{x}_{\text{sub}}(T)) \le H_0(\mathbf{x}_{\text{sub}}(0)) e^{-\rho T} + \frac{\beta}{\rho} \sup \|\mathbf{d}\|^2 + \frac{1}{2\rho} \sup \|\mathbf{v}\|^2,
	\]
	with \(\rho = \alpha - \frac{1}{L}\). Moreover, if \(\alpha > \frac{1}{L}\), then \(H_0\) remains bounded.
\end{theorem}

\begin{proof}
	By theorem \ref{ine}, let \(\mathbf{u} = \mathbf{h}(\mathbf{x}_{\text{sub}}) + \mathbf{v}\), where \(\mathbf{h}\) satisfies \eqref{condISS}. Then
	\[
	\begin{aligned}
		\frac{dH_0}{dt} &\le -\alpha H_0 + \beta \|\mathbf{d}\|^2 + (x_1 v_1 + x_2 v_2) \\
		&\le -\alpha H_0 + \beta \|\mathbf{d}\|^2 + \sqrt{x_1^2 + x_2^2} \|\mathbf{v}\| \\
		&\le -\alpha H_0 + \beta \|\mathbf{d}\|^2 + \sqrt{\frac{2}{L} H_0} \|\mathbf{v}\| \\
		&\le \left(-\alpha + \frac{1}{L}\right) H_0 + \beta \|\mathbf{d}\|^2 + \frac{1}{2} \|\mathbf{v}\|^2.
	\end{aligned}
	\]
	By the second comparison theorem, setting \(\rho = \alpha - \frac{1}{L}\), we obtain
	\[
	H_0(\mathbf{x}_{\text{sub}}(T)) \le H_0(\mathbf{x}_{\text{sub}}(0)) e^{-\rho T} + \frac{\beta}{\rho} \sup \|\mathbf{d}\|^2 + \frac{1}{2\rho} \sup \|\mathbf{v}\|^2.
	\]
	If \(\alpha > \frac{1}{L}\), then \(\rho > 0\) and as \(T \to +\infty\),
	\[
	H_0(\mathbf{x}_{\text{sub}}(T)) \le \frac{\beta}{\rho} \sup \|\mathbf{d}\|^2 + \frac{1}{2\rho} \sup \|\mathbf{v}\|^2.
	\]
\end{proof}

Theorem \ref{ROB} reveals the robustness of the feedback controller \(\mathbf{h}(\mathbf{x}_{\text{sub}})\). If \(\alpha > \frac{1}{L}\), then \(H_0(\mathbf{x}_{\text{sub}}(T)) = O(\|\mathbf{d}\|^2 + \|\mathbf{v}\|^2)\) as \(T \to +\infty\), meaning that \(\mathbf{h}(\mathbf{x}_{\text{sub}})\) has strong robustness against errors. Hence, the controller is feasible in practice.

Having established a robust feedback controller that ensures input‑to‑state stability and practical feasibility, we now turn to the numerical simulation of the closed‑loop system. Accurate long‑term simulation is essential for verifying control performance and predicting system behavior. However, standard numerical methods may artificially inject or dissipate energy, thereby obscuring the true physical response. To overcome this issue, the next section introduces a numerical algorithm that preserves the underlying Dirac structure of the port‑Hamiltonian system.

\section{A Dirac-structure-preserving numerical algorithm}\label{algorithm}

The Dirac structure is a geometric object that encodes how energy flows among the storage, resistive and external ports of a port‑Hamiltonian system~\cite{van2020port,kumar2025port}. Preserving this structure in numerical simulations is crucial for two reasons. First, it guarantees that the discrete‑time evolution respects the same energy balance relations as the continuous‑time system, thereby preventing unphysical energy drift over long simulation horizons. Second, it ensures that the qualitative properties of the controlled system, such as stability and passivity, are faithfully reproduced by the numerical scheme. Therefore, this section develops a Dirac‑structure‑preserving algorithm for the grid‑forming SVG model. The key idea is to replace the continuous‑time energy balance equation with a discrete counterpart that exactly conserves the Hamiltonian energy when external disturbances are absent.

\subsection{Dirac Structure of Port-Hamilton System}

\begin{definition}\label{Dirac}
	Let \(M\) be a smooth manifold and consider a vector bundle \(\mathbb{V}\) whose fibers are \(\mathbb{V}_{\mathbf{x}} = T_{\mathbf{x}} M \times T_{\mathbf{x}}^* M\), where \(T_{\mathbf{x}} M\) denotes the tangent space at \(\mathbf{x} \in M\) and \(T_{\mathbf{x}}^* M\) denotes the cotangent space. Define a bilinear form on \(\mathbb{V}_{\mathbf{x}}\) by
	\[
	\langle (X,\alpha), (Y,\beta) \rangle = \beta(X) + \alpha(Y),
	\]
	with \(X,Y \in T_{\mathbf{x}} M\) and \(\alpha,\beta \in T_{\mathbf{x}}^* M\). A Dirac structure \(\mathbb{D}\) is a smooth subbundle of \(\mathbb{V}\) such that each fiber \(\mathbb{D}_{\mathbf{x}}\) is a maximally isotropic subspace of \(\mathbb{V}_{\mathbf{x}}\), i.e.
	\[
	\mathbb{D}_{\mathbf{x}} = \{ \mathbf{v} \in \mathbb{V}_{\mathbf{x}} \mid \langle \mathbf{v}, \mathbf{w} \rangle = 0,\; \forall \mathbf{w} \in \mathbb{D}_{\mathbf{x}} \}.
	\]
\end{definition}

Lemma \ref{Dpro} illustrates the maximality of Dirac structure as isotropic subbundles, and theorem \ref{inD} shows that any isotropic subbundle is contained in a Dirac structure, thereby theoretically affirming the feasibility of designing Dirac-structure-preserving algorithms for port-Hamilton systems~\cite{van2020port,kumar2025port}.

\begin{lemma}\label{Dpro}
	Every fiber of a Dirac structure \(\mathbb{D}\) satisfies:
	\begin{itemize}
		\item[(a)] \(\langle \cdot,\cdot \rangle|_{\mathbb{D}_{\mathbf{x}}} = 0\);
		\item[(b)] There is no vector space \(\mathbb{E} \supseteq \mathbb{D}_{\mathbf{x}}\) such that \(\langle \cdot,\cdot \rangle|_{\mathbb{E}} = 0\);
		\item[(c)] \(\dim(\mathbb{D}_{\mathbf{x}}) = \dim(M)\).
	\end{itemize}
\end{lemma}

\begin{proof}
	Property (a) follows directly from definition \ref{Dirac}.
	
	If there is a vector space $\mathbb{E}\supsetneqq \mathbb{D}_\mathbf{x}$ such that $\langle \cdot,\cdot\rangle\big|_{\mathbb{E}}=0$, then there is some $\mathbf{v}\notin\mathbb{D}_\mathbf{x}$ such that $\langle \mathbf{v},\mathbf{w}\rangle=0,\forall \mathbf{w}\in\mathbb{D}_\mathbf{x}$, contradicting definition \ref{Dirac}.
	
	Suppose $\dim M=n$, then the bilinear form on $\mathbb{V}_\mathbf{x}$ can be represented by $\begin{bmatrix}
		0&I_n\\I_n&0
	\end{bmatrix}$. Set $\mathbb{D}_\mathbf{x}^\perp=\{\mathbf{v}\in\mathbb{V}\big| \langle\mathbf{v},\mathbf{w}\rangle=0,\forall \mathbf{w}\in\mathbb{D}_\mathbf{x}\}$, then by the non-singularity of $\begin{bmatrix}
	0&I_n\\I_n&0
	\end{bmatrix}$, $\dim(\mathbb{D}_\mathbf{x}^\perp)=\dim(\mathbb{V}_\mathbf{x})-\dim(\mathbb{D}_\mathbf{x})=2n-\dim(\mathbb{D}_\mathbf{x})$. Moreover, by definition \ref{Dirac} $\mathbb{D}_\mathbf{x}^\perp=\mathbb{D}_\mathbf{x}$, so $\dim(\mathbb{D}_\mathbf{x})=2n-\dim(\mathbb{D}_\mathbf{x})$. The solution of this equation is $\dim(\mathbb{D}_\mathbf{x})=n$.
\end{proof}

\begin{theorem}\label{inD}
	If \(\mathbb{L} \subset \mathbb{V}_{\mathbf{x}}\) satisfies \(\langle \cdot,\cdot \rangle|_{\mathbb{L}} = 0\), then there exists a subspace \(\mathbb{D}_{\mathbf{x}} = \{ \mathbf{v} \in \mathbb{V}_{\mathbf{x}} \mid \langle \mathbf{v}, \mathbf{w} \rangle = 0,\; \forall \mathbf{w} \in \mathbb{D}_{\mathbf{x}} \}\) such that \(\mathbb{L} \subset \mathbb{D}_{\mathbf{x}}\).
\end{theorem}

\begin{proof}
    From lemma \ref{Dpro} a Dirac structure is a maximum isotropic subbundle whose fibers have the same dimension as the origin manifold $M$. Now prove the existence of the maximum isotropic subbundle, and then prove that the maximum isotropic subbundle is the Dirac structure from the perspective of dimension.

	Let $\mathcal{B}=\{\mathbb{E}_\mathbf{x}\subset\mathbb{V}_\mathbf{x}\big| \mathbb{L}\subset\mathbb{E}_\mathbf{x},\langle\cdot,\cdot\rangle\big|_{\mathbb{E}_\mathbf{x}}=0\}$ be a family of subspace, and let $\subset$ be the partial relation. For each chain $\mathcal{C}\subset\mathcal{B}$ from the partial relation, suppose $\mathbb{C}_{\mathbf{x}}=\bigcup\limits_{\mathbb{E}_\mathbf{x}\in\mathcal{C}}\mathbb{E}_{\mathbf{x}}$, then $\mathbb{C}_\mathbf{x}\subset\mathbb{V}_\mathbf{x}$ is a vector space satisfying that $\langle\cdot,\cdot\rangle\big|_{\mathbb{C}_\mathbf{x}}=0$. So $\mathbb{C}_\mathbf{x}$ is the upper bound of $\mathcal{C}$, which means that every chain has an upper bound. By Zorn's lemma, there is a maximal element $\mathbb{D}_\mathbf{x}\in\mathcal{B}$.
	
	Since $\mathbb{D}_\mathbf{x}\in\mathcal{B}$, we see $\mathbb{D}_\mathbf{x}^\perp=\{\mathbf{v}\in \mathbb{V}_\mathbf{x}\big| \langle \mathbf{v},\mathbf{w}\rangle=0,\forall \mathbf{w}\in \mathbb{D}_\mathbf{x}\}\supset\mathbb{D}_\mathbf{x}$. If $\mathbb{D}_\mathbf{x}^\perp\neq\mathbb{D}_\mathbf{x}$, suppose $\dim(M)=n$, then $\dim(\mathbb{D}_\mathbf{x}^\perp)=2n-\dim(\mathbb{D}_\mathbf{x})>\dim(\mathbb{D}_\mathbf{x})$, which means $\dim(\mathbb{D}_\mathbf{x})<n$ and $\dim(\mathbb{D}_\mathbf{x}^\perp)>n$. Moreover, if $\mathbf{v}=(X,\mathbf{0})$ in which $X\in T_\mathbf{x}M$ satisfies $\alpha(X)=0,\forall \mathbf{w}=(Y,\alpha)\in\mathbb{D}_x$, then $\mathbf{v}\in\mathbb{D}_\mathbf{x}^\perp$ and $\langle\mathbf{v},\mathbf{v}\rangle=0$. Set $\mathbb{W}=\{\mathbf{v}=(X,\mathbf{0})\big| \alpha(X)=0,\forall \mathbf{w}=(Y,\alpha)\in\mathbb{D}_\mathbf{x}\}$, then $\dim(\mathbb{W})+\dim(p(\mathbb{D}_\mathbf{x}))=n$, in which $p((Y,\alpha))=(0,\alpha)$. If $\mathbb{W}\subset\mathbb{D}_\mathbf{x}$, then $\dim(\mathbb{W})+\dim(p(\mathbb{D}_\mathbf{x}))\leq\dim(\mathbb{D}_\mathbf{x})<n$, which is contradictory with $\dim(\mathbb{W})+\dim(p(\mathbb{D}_\mathbf{x}))=n$. So there is a vector $\mathbf{v}\in\mathbb{W}\backslash\mathbb{D}_\mathbf{x}$. However,$\mathbf{v}$ satisfies $\langle\mathbf{v},\mathbf{v}\rangle=0$ and $\mathbf{v}\in\mathbb{D}_\mathbf{x}^\perp\backslash\mathbb{D}_\mathbf{x}$, which is contradictory with the maximality of $\mathbb{D}_\mathbf{x}$. So $\mathbb{D}_{\mathbf{x}}=\mathbb{D}_\mathbf{x}^\perp=\{\mathbf{v}\in \mathbb{V}_\mathbf{x}\big| \langle \mathbf{v},\mathbf{w}\rangle=0,\forall \mathbf{w}\in \mathbb{D}_\mathbf{x}\}$.
\end{proof}

In order to facilitate the study of control systems, the extended Dirac structure is defined below~\cite{van2020port,kong2023control}.

\begin{definition}\label{extDirac}
	Let \(M\) be a smooth manifold and consider an extended vector bundle \(\mathbb{V}^{\text{ext}}\) with fibers
	\[
	\mathbb{V}_{\mathbf{x}}^{\text{ext}} = T_{\mathbf{x}} M \times T_{\mathbf{x}}^* M \times \mathbb{R}^m \times (\mathbb{R}^m)^*.
	\]
	Define a bilinear form on \(\mathbb{V}_{\mathbf{x}}^{\text{ext}}\) by
	\[
	\langle (\mathbf{f}_1,\mathbf{e}_1,\mathbf{u}_1,\mathbf{y}_1), (\mathbf{f}_2,\mathbf{e}_2,\mathbf{u}_2,\mathbf{y}_2) \rangle = \mathbf{e}_2(\mathbf{f}_1) + \mathbf{y}_2(\mathbf{u}_1) + \mathbf{e}_1(\mathbf{f}_2) + \mathbf{y}_1(\mathbf{u}_2).
	\]
	An extended Dirac structure \(\mathbb{D}^{\text{ext}}\) is a smooth subbundle of \(\mathbb{V}^{\text{ext}}\) such that each fiber \(\mathbb{D}_{\mathbf{x}}^{\text{ext}}\) is a maximally isotropic subspace of $\mathbb{V}_\mathbf{x}^{\text{ext}}$, i.e.\[\mathbb{D}_{\mathbf{x}}^{\text{ext}} = \left\{ \mathbf{v} \in \mathbb{V}_{\mathbf{x}}^{\text{ext}} \mid \langle \mathbf{v}, \mathbf{w} \rangle = 0,\; \forall \mathbf{w} \in \mathbb{D}_{\mathbf{x}}^{\text{ext}} \right\}.\]
\end{definition}

From the same proof as the proof of lemma \ref{Dpro} and theorem \ref{inD} we can find the lemma and the theorem below:

\begin{lemma}\label{extDpro}
	Every fiber of an extended Dirac structure satisfies:
	\begin{itemize}
		\item[(a)] \(\langle \cdot,\cdot \rangle|_{\mathbb{D}_{\mathbf{x}}^{\text{ext}}} = 0\);
		\item[(b)] No proper superspace has zero bilinear form;
		\item[(c)] \(\dim(\mathbb{D}_{\mathbf{x}}^{\text{ext}}) = \dim(M) + m\).
	\end{itemize}
\end{lemma}

\begin{theorem}\label{inDext}
	If \(\mathbb{L} \subset \mathbb{V}_{\mathbf{x}}^{\text{ext}}\) satisfies \(\langle \cdot,\cdot \rangle|_{\mathbb{L}} = 0\), then there exists a subspace \(\mathbb{D}_{\mathbf{x}}^{\text{ext}}\) containing \(\mathbb{L}\) such that \(\mathbb{D}_{\mathbf{x}}^{\text{ext}} = \{\mathbf{v} \mid \langle \mathbf{v}, \mathbf{w} \rangle = 0,\; \forall \mathbf{w} \in \mathbb{D}_{\mathbf{x}}^{\text{ext}}\}\).
\end{theorem}

For the port-Hamiltonian system \eqref{PH} with \(R(\mathbf{x}) = 0\), the vectors
\[
\mathbf{v} = \left(-\dot{\mathbf{x}}^1, \nabla H(\mathbf{x}), \mathbf{u}^1, \mathbf{d}^1, C(\mathbf{x})^\top \nabla H(\mathbf{x}), B(\mathbf{x}) \nabla H(\mathbf{x})\right)
\]
and similarly \(\mathbf{w}\) satisfy \(\langle \mathbf{v}, \mathbf{w} \rangle = 0\) because
\[
\nabla H(\mathbf{x})^\top \dot{\mathbf{x}} = \nabla H(\mathbf{x})^\top C(\mathbf{x}) \mathbf{u} + \nabla H(\mathbf{x})^\top B(\mathbf{x}) \mathbf{d}
\]
for all \(\mathbf{u} \in \mathbb{R}^m, \mathbf{d} \in \mathbb{R}^k\). By Theorem \ref{inDext}, there exists a Dirac structure \(\mathbb{D}\) such that
\[
\mathbb{W} = \left\{ (-\dot{\mathbf{x}}, \nabla H, \mathbf{u}, \mathbf{d}, C^\top \nabla H, B \nabla H) \mid \mathbf{u} \in \mathbb{R}^m, \mathbf{d} \in \mathbb{R}^k \right\} \subset \mathbb{D}_{\mathbf{x}}
\]
for all \(\mathbf{x} \in \mathbb{R}^n\). This reveals the relationship between Dirac structures and power conservation. Therefore, in order to better simulate the energy variation of a port-Hamilton system, the Dirac structure can be discretized, and then an algorithm that satisfies the discrete Dirac structure can be designed.

\subsection{An Algorithm Preserving Dirac Structure}

We now focus on the Dirac structure of a port-Hamiltonian system without damping (i.e., \(R(\mathbf{x}) = 0\)). The condition that \(\mathbb{W}\) is contained in a Dirac structure is equivalent to
\[
\nabla H(\mathbf{x})^\top \dot{\mathbf{x}} = \nabla H(\mathbf{x})^\top C(\mathbf{x}) \mathbf{u} + \nabla H(\mathbf{x})^\top B(\mathbf{x}) \mathbf{d}
\]
for all \(\mathbf{u}, \mathbf{d}\).

Let \(\mathbf{x}_k = \mathbf{x}(kh)\), \(\mathbf{x}_{k+1} = \mathbf{x}((k+1)h)\) with step size \(h > 0\). Integrating the above equality gives \begin{align}\label{int}
	H(\mathbf{x}_{k+1})-H(\mathbf{x}_k)=\int_{kh}^{(k+1)h}\nabla H^\top C(\mathbf{x}(t))\mathbf{u}(t)+\nabla H^\top B(\mathbf{x}(t))\mathbf{d}(t))dt.
\end{align}

On the interval \((kh, (k+1)h]\), \(\mathbf{u}(t)\) is constant, and \(\mathbf{d}(t)\) is unknown in practice. Hence we approximate them by \(\mathbf{u}_k\) and \(\mathbf{d}_k\).

Now discretize the Dirac structure of a port-Hamilton system based on the characteristics of power variation~\cite{hairer2006geometric,reich1996enhancing}.

If there exists a function \(\overline{\nabla H}: \mathbb{R}^n \times \mathbb{R}^n \to \mathbb{R}^n\) such that
\[
H(\mathbf{x}_{k+1}) - H(\mathbf{x}_k) = \overline{\nabla H}(\mathbf{x}_k, \mathbf{x}_{k+1})^\top (\mathbf{x}_{k+1} - \mathbf{x}_k),
\]
then equation \eqref{int} can be approximated by
\[
\overline{\nabla H}(\mathbf{x}_k, \mathbf{x}_{k+1})^\top (\mathbf{x}_{k+1} - \mathbf{x}_k) = h \overline{\nabla H}(\mathbf{x}_k, \mathbf{x}_{k+1})^\top \bigl( C_{k,k+1} \mathbf{u}_k + B_{k,k+1} \mathbf{d}_k \bigr),
\]
where \(C_{k,k+1}\) and \(B_{k,k+1}\) are matrices such that \(h(C_{k,k+1} \mathbf{u}_k + B_{k,k+1} \mathbf{d}_k)\) approximates the integral on the right-hand side of \eqref{int}.

\begin{remark}
	If \(C(\mathbf{x})^\top \nabla H(\mathbf{x}) = \mathbf{0}\) or \(B(\mathbf{x})^\top \nabla H(\mathbf{x}) = \mathbf{0}\), then in order to preserve the property that control inputs or disturbance inputs have no effect on energy, \(C_{k,k+1}\) or \(B_{k,k+1}\) has to satisfy \(C_{k,k+1}^\top \overline{\nabla H}(\mathbf{x}_k,\mathbf{x}_{k+1}) = \mathbf{0}\) or \(B_{k,k+1}^\top \overline{\nabla H}(\mathbf{x}_k,\mathbf{x}_{k+1}) = \mathbf{0}\), respectively.
\end{remark}

\begin{definition}
	Given a port-Hamiltonian system \(\dot{\mathbf{x}} = J(\mathbf{x}) \nabla H(\mathbf{x}) + B(\mathbf{x}) \mathbf{d} + C(\mathbf{x}) \mathbf{u}\), suppose there exists a function \(\overline{\nabla H}\) satisfying the discrete energy difference. Then an algorithm \(\mathbf{x}_{k+1} = \phi(\mathbf{x}_k)\) is called \textbf{Dirac-structure-preserving} if
	\[
	\overline{\nabla H}(\mathbf{x}_k, \mathbf{x}_{k+1})^\top (\mathbf{x}_{k+1} - \mathbf{x}_k) = h \overline{\nabla H}(\mathbf{x}_k, \mathbf{x}_{k+1})^\top \bigl( C_{k,k+1} \mathbf{u}_k + B_{k,k+1} \mathbf{d}_k \bigr)
	\]
	holds for all \(\mathbf{u}_k \in \mathbb{R}^m\) and \(\mathbf{d}_k \in \mathbb{R}^k\).
\end{definition}

For the grid-forming SVG system \eqref{PHMod}, we have
\[
H(\mathbf{x}_{k+1}) - H(\mathbf{x}_k) =  \nabla H\left(\frac{\mathbf{x}_{k+1} + \mathbf{x}_k}{2}\right)^\top (\mathbf{x}_{k+1} - \mathbf{x}_k).
\]
Since \(C(\mathbf{x})^\top \nabla H(\mathbf{x}) = 0\) and \(B\) is constant, we set \(B_{k,k+1} = B\). Hence a Dirac-structure-preserving algorithm must satisfy
\[
\nabla H\left(\frac{\mathbf{x}_{k+1} + \mathbf{x}_k}{2}\right)^\top (\mathbf{x}_{k+1} - \mathbf{x}_k) = h \nabla H\left(\frac{\mathbf{x}_{k+1} + \mathbf{x}_k}{2}\right)^\top B \mathbf{d}_k.
\]
If \(\mathbf{d}_k = 0\), then \(H(\mathbf{x}_{k+1}) = H(\mathbf{x}_k)\), which is consistent with \(\frac{dH}{dt} = 0\) in system \eqref{PHMod}.

A simple implementation is the midpoint rule~\cite{hairer2006geometric,reich1996enhancing,hairer2011geometric}:
\begin{equation}
	\frac{\mathbf{x}_{k+1} - \mathbf{x}_k}{h} = J \nabla H\!\left(\frac{\mathbf{x}_{k+1} + \mathbf{x}_k}{2}\right) + B \mathbf{d}_k + C\!\left(\frac{\mathbf{x}_{k+1} + \mathbf{x}_k}{2}\right) \mathbf{u}_k, \label{eq:midpoint}
\end{equation}
which has truncation error \(O(h^3)\) when \(\mathbf{d}_k = \mathbf{0}\).

Solving \eqref{eq:midpoint} yields
\[
\mathbf{x}_{k+1} = \left( I - \frac{h}{2} J_0 + \frac{h}{2} K_0 \right)^{-1} \left( \left( I + \frac{h}{2} J_0 - \frac{h}{2} K_0 \right) \mathbf{x}_k + h B \mathbf{d}_k + h C_0 \mathbf{u}_k \right),
\]
where
\[
J_0 = \begin{bmatrix}
	0 & \omega & -\frac{1}{L} & 0 & 0 \\
	-\omega & 0 & 0 & -\frac{1}{L} & 0 \\
	\frac{1}{C} & 0 & 0 & \omega & 0 \\
	0 & \frac{1}{C} & -\omega & 0 & 0 \\
	0 & 0 & 0 & 0 & 0
\end{bmatrix}, 
K_0 = \begin{bmatrix}
	0 & 0 & 0 & 0 & 0 \\
	0 & 0 & 0 & 0 & 0 \\
	0 & 0 & 0 & 0 & 0 \\
	0 & 0 & 0 & 0 & 0 \\
	u_1 & u_2 & 0 & 0 & 0
\end{bmatrix}, 
C_0 = \begin{bmatrix}
	\frac{1}{L} & 0 \\
	0 & \frac{1}{L} \\
	0 & 0 \\
	0 & 0 \\
	0 & 0
\end{bmatrix}.
\]

The above scheme preserves the Dirac structure and exactly conserves the Hamiltonian energy when external disturbances are absent.

\section{Numerical Experiments}\label{numexp}
The previous section introduced a Dirac‑structure‑preserving numerical algorithm that exactly conserves energy in the absence of disturbances. To validate its practical advantages, we now compare it against a conventional Runge‑Kutta method. Additionally, the ISS‑based controller designed in Section \ref{con} is tested against a standard PI controller. For SVG system \eqref{SVG}, set \(L = C = \omega = 1\). This section examines the performance of the proposed control method and the structure-preserving algorithm.

\subsection{Comparison of Different Control Methods}\label{comparison}
The controlled subsystem is
\begin{equation}
	\frac{d}{dt} \begin{bmatrix} x_1 \\ x_2 \\ x_3 \\ x_4 \end{bmatrix} = \begin{bmatrix}
		0 & 1 & -1 & 0 \\
		-1 & 0 & 0 & -1 \\
		1 & 0 & 0 & 1 \\
		0 & 1 & -1 & 0
	\end{bmatrix} \begin{bmatrix} x_1 \\ x_2 \\ x_3 \\ x_4 \end{bmatrix} - \begin{bmatrix} 0 \\ 0 \\ i_g^d \\ i_g^q \end{bmatrix} + \begin{bmatrix} 1 & 0 \\ 0 & 1 \\ 0 & 0 \\ 0 & 0 \end{bmatrix} \begin{bmatrix} u_1 \\ u_2 \end{bmatrix}. \label{ex}
\end{equation}

Assuming \(u_1, u_2\) are constant over each time step, the exact solution can be obtained via matrix exponentials. For \(i_g^d = i_g^q = 0\),
\[
\mathbf{x}(t) = \Phi(t) \mathbf{x}(0) + \Psi(t) \begin{bmatrix} u_1 \\ u_2 \end{bmatrix},
\]
with explicit expressions given in the original paper. For \(i_g^d(t) = \cos 2t,\; i_g^q(t) = \sin 2t\), similar explicit formulas exist.

First, a PI controller is applied:
\[
\mathbf{u} = -K_p \begin{bmatrix} x_1 \\ x_2 \end{bmatrix} - K_i \begin{bmatrix} \int_0^t x_1(\tau) d\tau \\ \int_0^t x_2(\tau) d\tau \end{bmatrix},
\]
where \(K_p, K_i\) are determined via LQR with \(Q = \operatorname{diag}(0,0,10,10,1,1)\), \(R = I_2\). This yields approximately
\[
K_p \approx \begin{bmatrix} 2.1956 & -0.8878 \\ 0.8878 & 2.1956 \end{bmatrix}, \quad
K_i \approx \begin{bmatrix} 0.8364 & 0.5481 \\ -0.5481 & 0.8364 \end{bmatrix}.
\]
The integral is approximated by the trapezoidal rule.

Second, our ISS-based controller from Section~\ref{con} is used. With \(\alpha = 2\), \(\epsilon = 0.125\),
\[
\begin{cases}
	u_1 = -x_1 \left( 2\dfrac{x_3^2 + x_4^2}{x_1^2 + x_2^2} + 1 \right), \\[1em]
	u_2 = -x_2 \left( 2\dfrac{x_3^2 + x_4^2}{x_1^2 + x_2^2} + 1 \right),
\end{cases}
\]
where the ratio is capped at 5 when large.

Figures \ref{fig:control1} and \ref{fig:control2} compare the PI controller and the proposed ISS-based controller under no disturbance (Fig. \ref{fig:control1}) and under a periodic disturbance $i_g^d=\cos 2t,i_g^q=\sin 2t$ (Fig. \ref{fig:control2}), respectively. In both cases, the ISS controller achieves superior performance. Regarding steady‑state offset: under no disturbance, both controllers drive all states exactly to zero; under disturbance, the PI controller produces an oscillatory offset with amplitude $\approx 0.5$ around a non‑zero mean, whereas the ISS controller reduces the offset to about 0.6. In terms of settling time, the PI controller takes approximately 12s (Fig. \ref{fig:control1}) and 17s (Fig. \ref{fig:control2}), while the ISS controller settles significantly faster, demonstrating its advantage in both transient and steady‑state behavior. Moreover, the magnitude of the control input required by the ISS controller is consistently lower than that of the PI controller under the same conditions, indicating better energy efficiency. Beyond these quantitative advantages, the ISS controller is not a black‑box design: it has only three tunable parameters ($\alpha,\epsilon$ and a saturation bound), each with a clear physical interpretation. $\alpha$ governs the convergence speed, $\epsilon=\frac{\beta}{\alpha}$ determines the ultimate bound of $H_0$, and the saturation bound directly limits the control effort. This transparency makes parameter tuning straightforward and much less demanding than the trial‑and‑error or LQR‑based tuning required for the PI controller. Hence, the ISS controller is both more effective and easier to implement in practice.

\begin{figure}[H]
	\centering
	\includegraphics[width=0.49\textwidth]{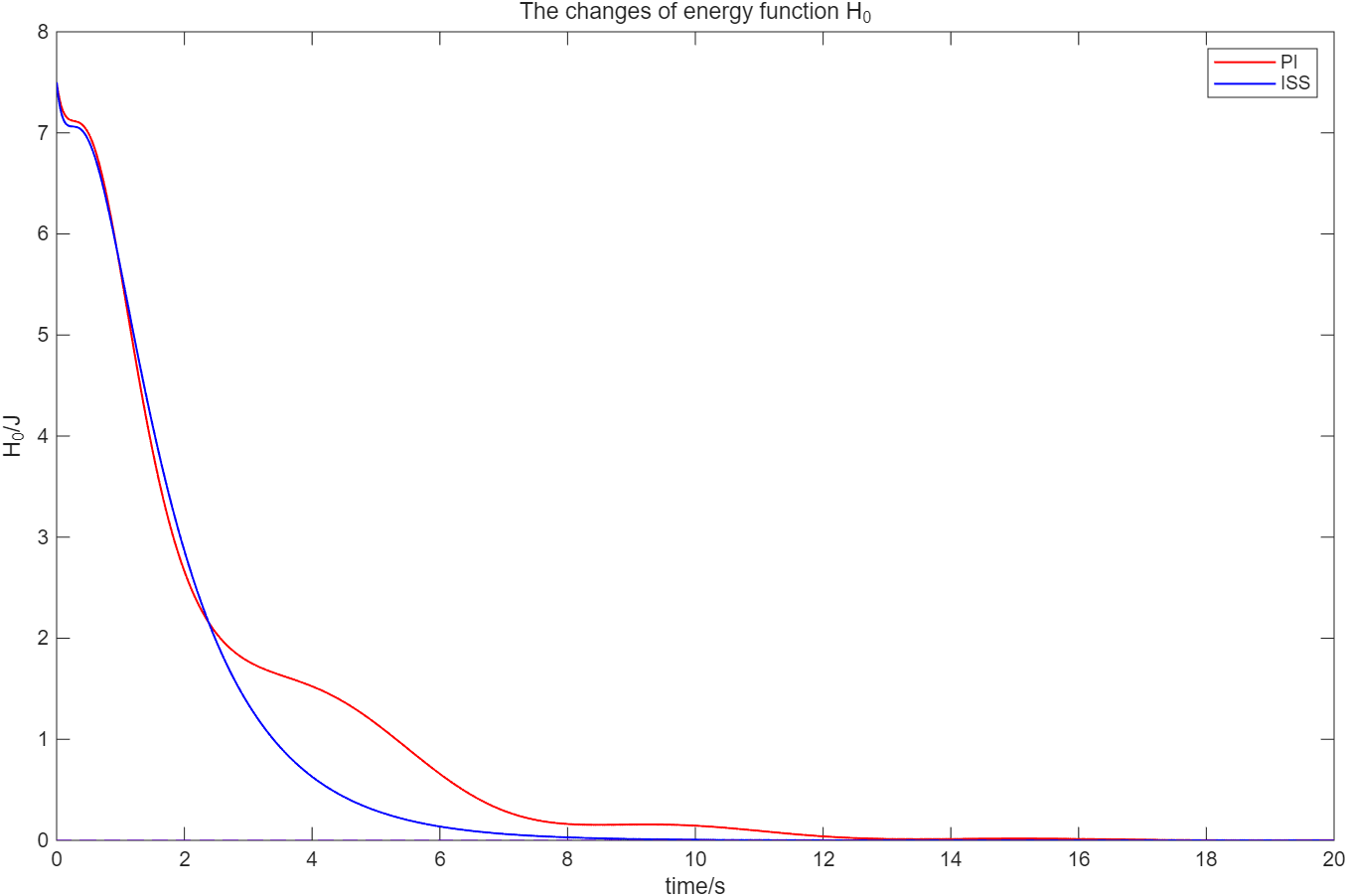}
	\includegraphics[width=0.49\textwidth]{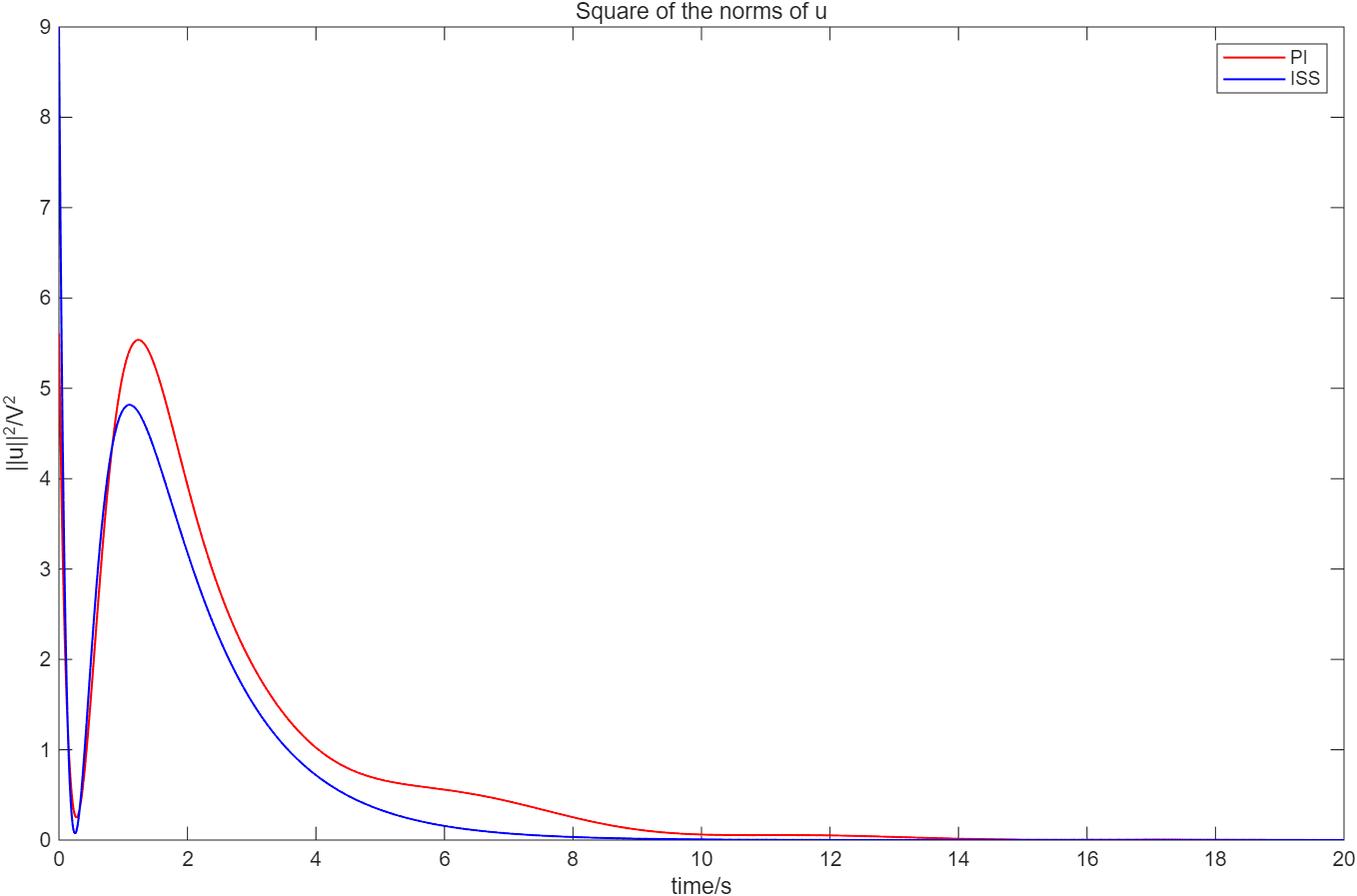}
	\includegraphics[width=0.49\textwidth]{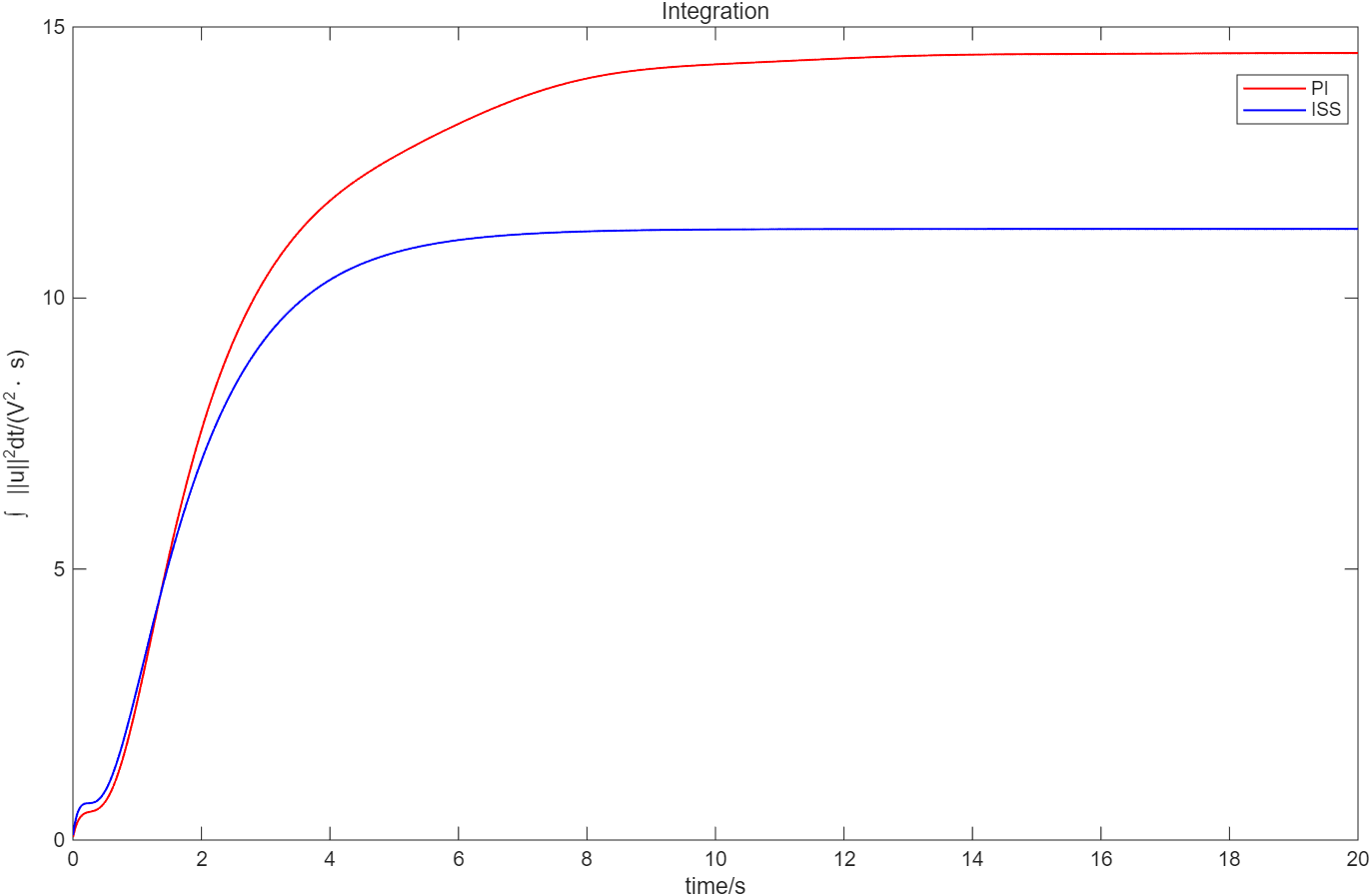}
	\caption{The comparison of two controllers as $\mathbf{d}=\mathbf{0}$.}
	\label{fig:control1}
\end{figure}

\begin{figure}[H]
	\centering
	\includegraphics[width=0.49\textwidth]{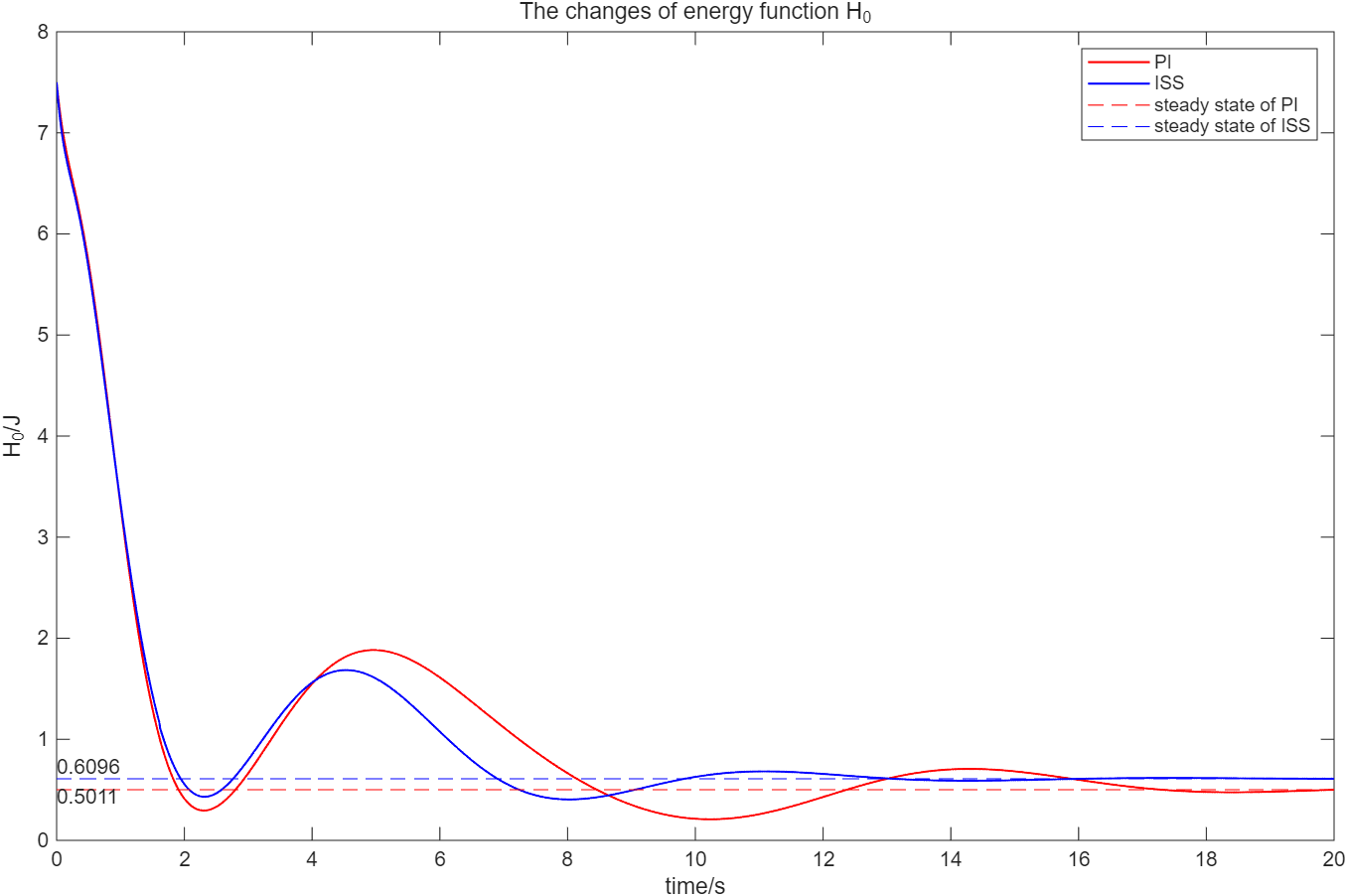}
	\includegraphics[width=0.49\textwidth]{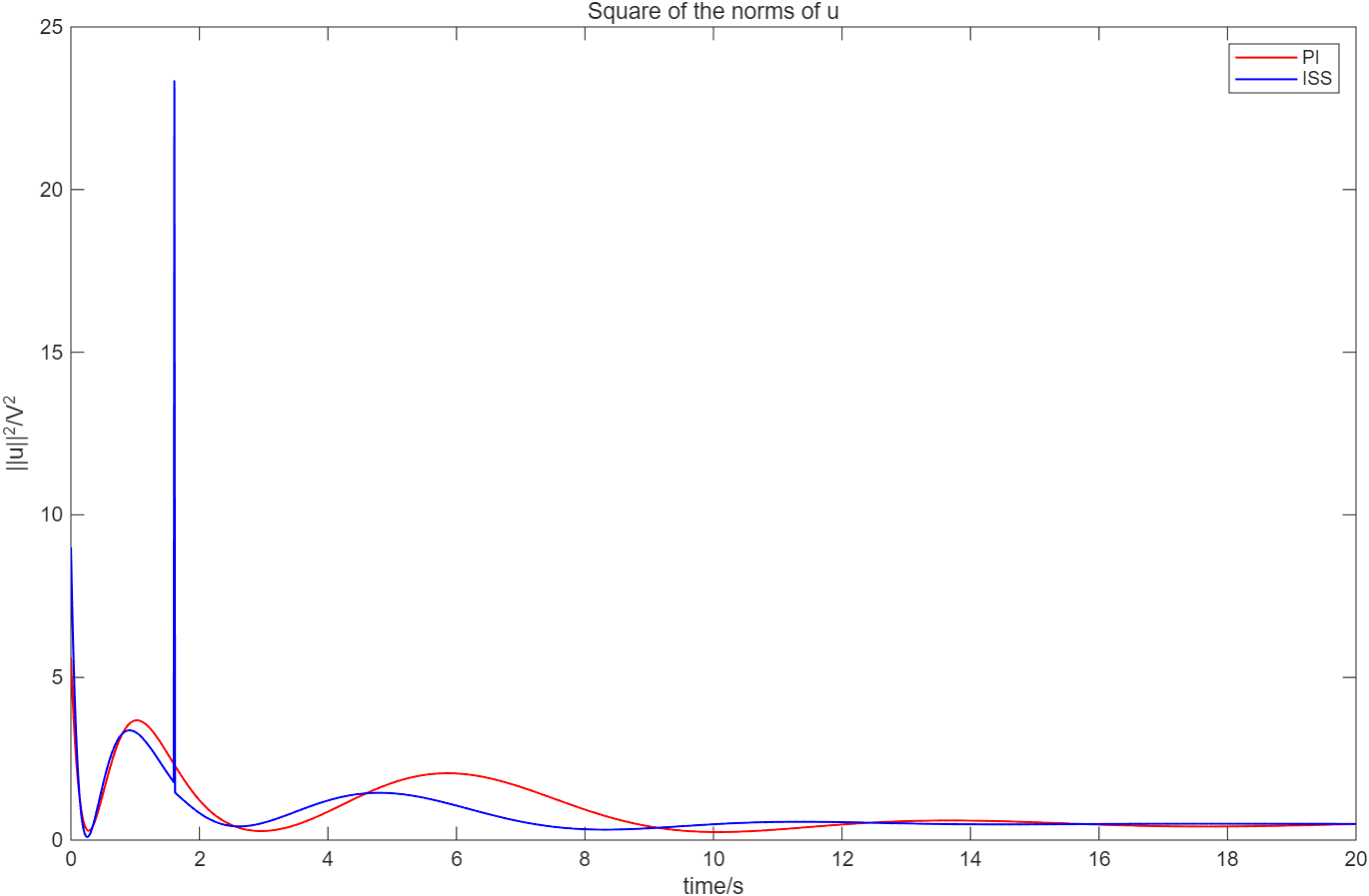}
	\includegraphics[width=0.49\textwidth]{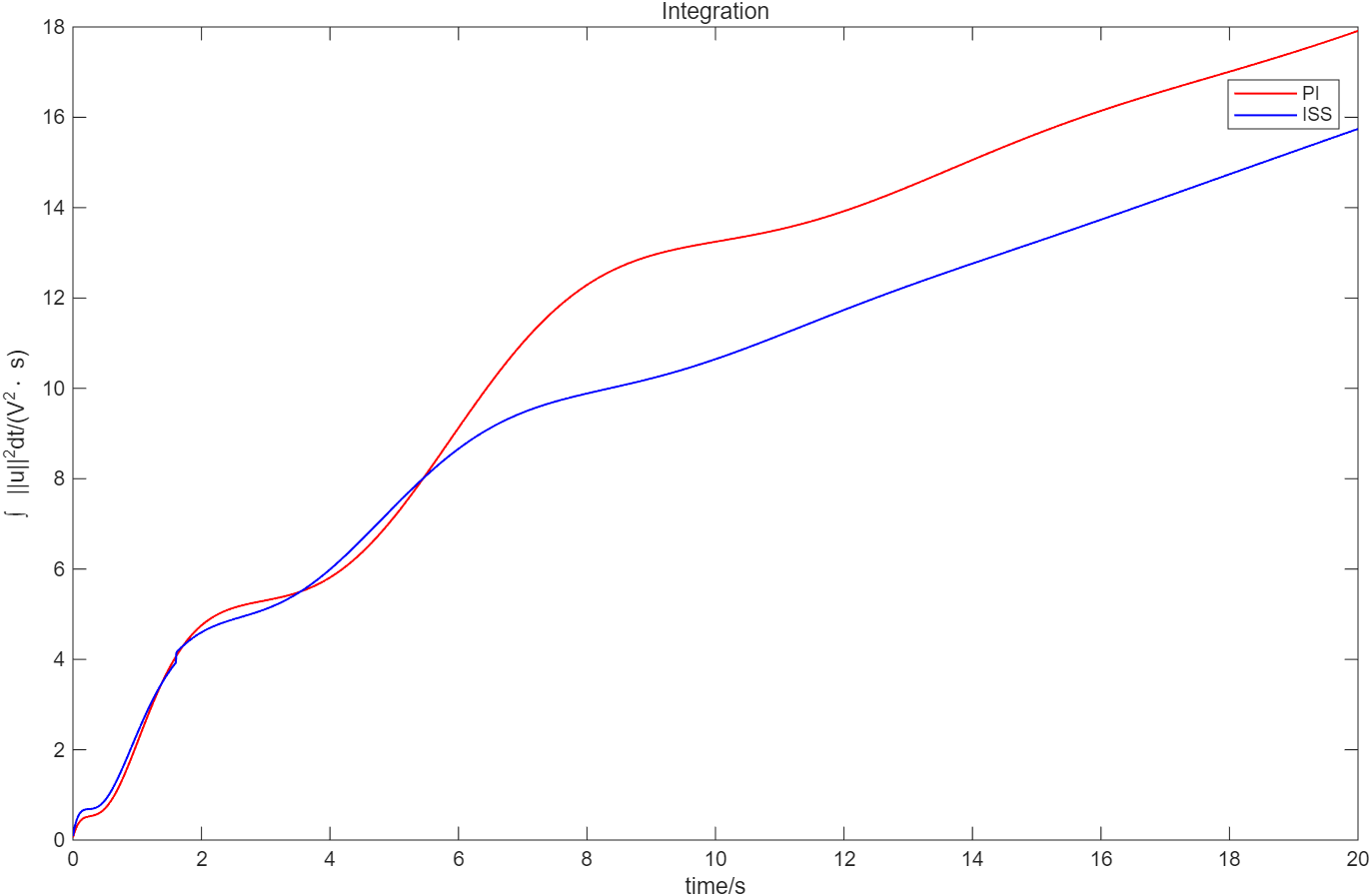}
	\caption{The comparison of two controllers as $\mathbf{d}=\begin{bmatrix}
			\cos 2t\\\sin 2t
		\end{bmatrix}$.}
	\label{fig:control2}
\end{figure}

\subsection{Comparison of Different Algorithms}
We now compare two algorithms with the same truncation error $O(h^3)$ on the system
\begin{equation}
	\frac{d}{dt} \begin{bmatrix} x_1 \\ x_2 \\ x_3 \\ x_4\\x_5 \end{bmatrix} = \begin{bmatrix}
		0 & 1 & -1 & 0 &0\\
		-1 & 0 & 0 & -1 &0\\
		1 & 0 & 0 & 1&0 \\
		0 & 1 & -1 & 0&0\\
		0&0&0&0&0
	\end{bmatrix} \begin{bmatrix} x_1 \\ x_2 \\ x_3 \\ x_4\\x_5 \end{bmatrix} + \begin{bmatrix} 1 & 0 \\ 0 & 1 \\ 0 & 0 \\ 0 & 0\\-x_1&-x_2 \end{bmatrix} \begin{bmatrix} u_1 \\ u_2 \end{bmatrix}. \label{eq:exp_sys2}
\end{equation}
Set \(\mathbf{d} = \mathbf{0}\) and use the ISS controller from Section~\ref{con}. The exact solution is given by a known matrix exponential.

\begin{itemize}
	\item \textbf{Algorithm A (non-preserving):} Two-stage second-order method:
	\begin{align*}
	\mathbf{x}_{k+1}&=\mathbf{x}_k+\frac{h}{2}(\mathbf{k}_1+\mathbf{k}_2),\\
	\mathbf{k}_1&=\mathbf{f}\left(\mathbf{x}_k+\frac{h}{4}\mathbf{k}_1\right),\\    \mathbf{k}_2&=\mathbf{f}\left(\mathbf{x}_k+h\left(\mathbf{k}_2-\frac{1}{4}\mathbf{k}_1\right)\right).
	\end{align*}
	\item \textbf{Algorithm B (preserving):} The midpoint method \begin{equation*}
	\frac{\mathbf{x}_{k+1} - \mathbf{x}_k}{h} = J \nabla H\!\left(\frac{\mathbf{x}_{k+1} + \mathbf{x}_k}{2}\right) + B \mathbf{d}_k + C\!\left(\frac{\mathbf{x}_{k+1} + \mathbf{x}_k}{2}\right) \mathbf{u}_k
\end{equation*} that preserves the Dirac structure.
\end{itemize}

Figure~\ref{fig:algorithm} shows the comparison. The Dirac‑structure‑preserving midpoint method exactly conserves the Hamiltonian energy $H$ when $\mathbf{d}=\mathbf{0}$, i.e., $H(\mathbf{x}_{k+1})=H(\mathbf{x}_k)$ for all steps, whereas the Runge‑Kutta method exhibits a small but steadily growing energy drift. Furthermore, the error of the preserving algorithm is insensitive to the magnitude of the control input $\mathbf{u}$: even when $||\mathbf{u}||$ is large, the energy error remains exactly zero and the trajectory error stays bounded. In contrast, the Runge‑Kutta method shows error amplification proportional to $||\mathbf{u}||$. This insensitivity makes the Dirac‑preserving algorithm particularly suitable for long‑term simulations where strong control actions are applied.

\begin{figure}[H]
	\centering
	\includegraphics[width=0.49\textwidth]{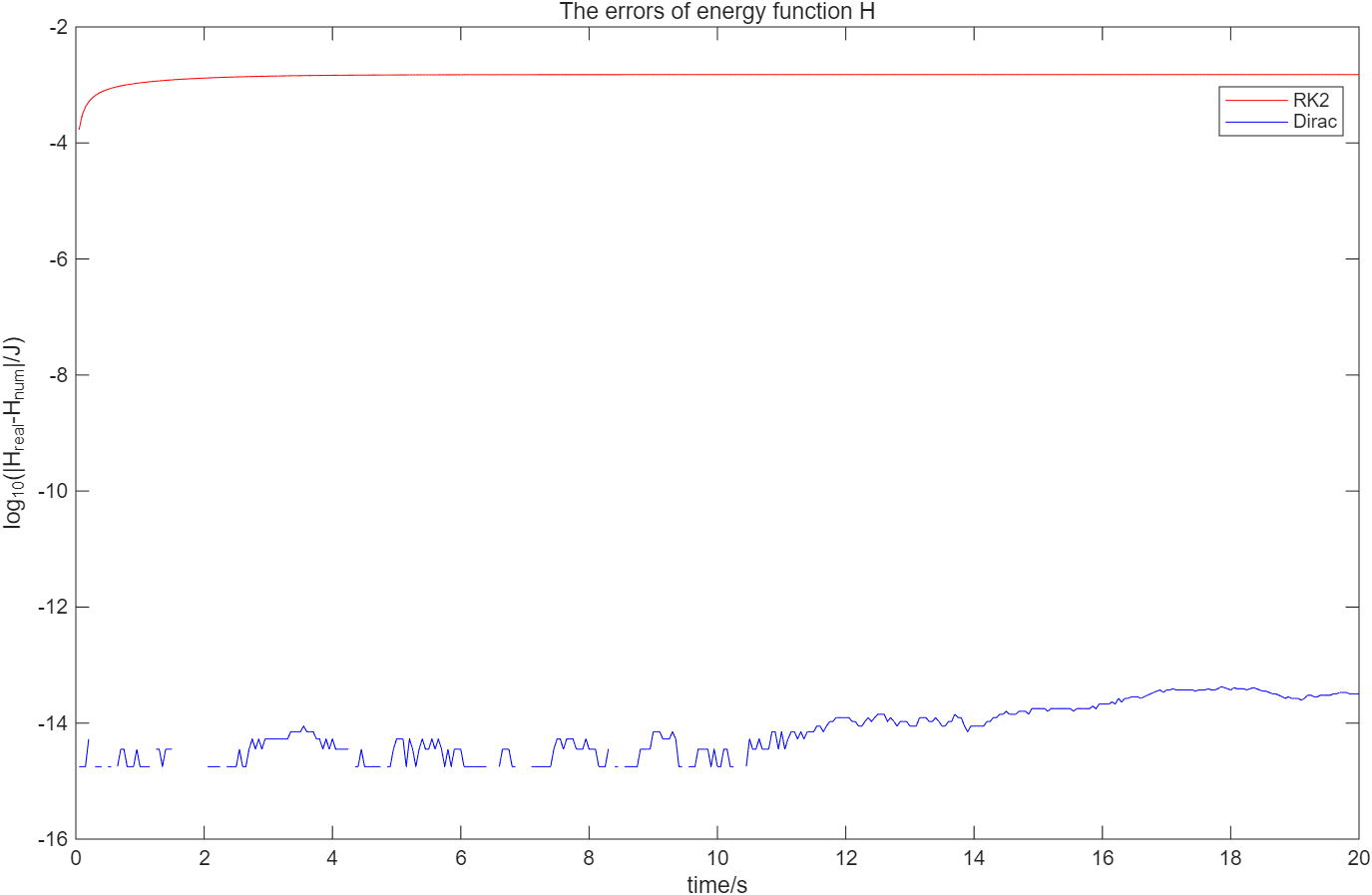}
	\includegraphics[width=0.49\textwidth]{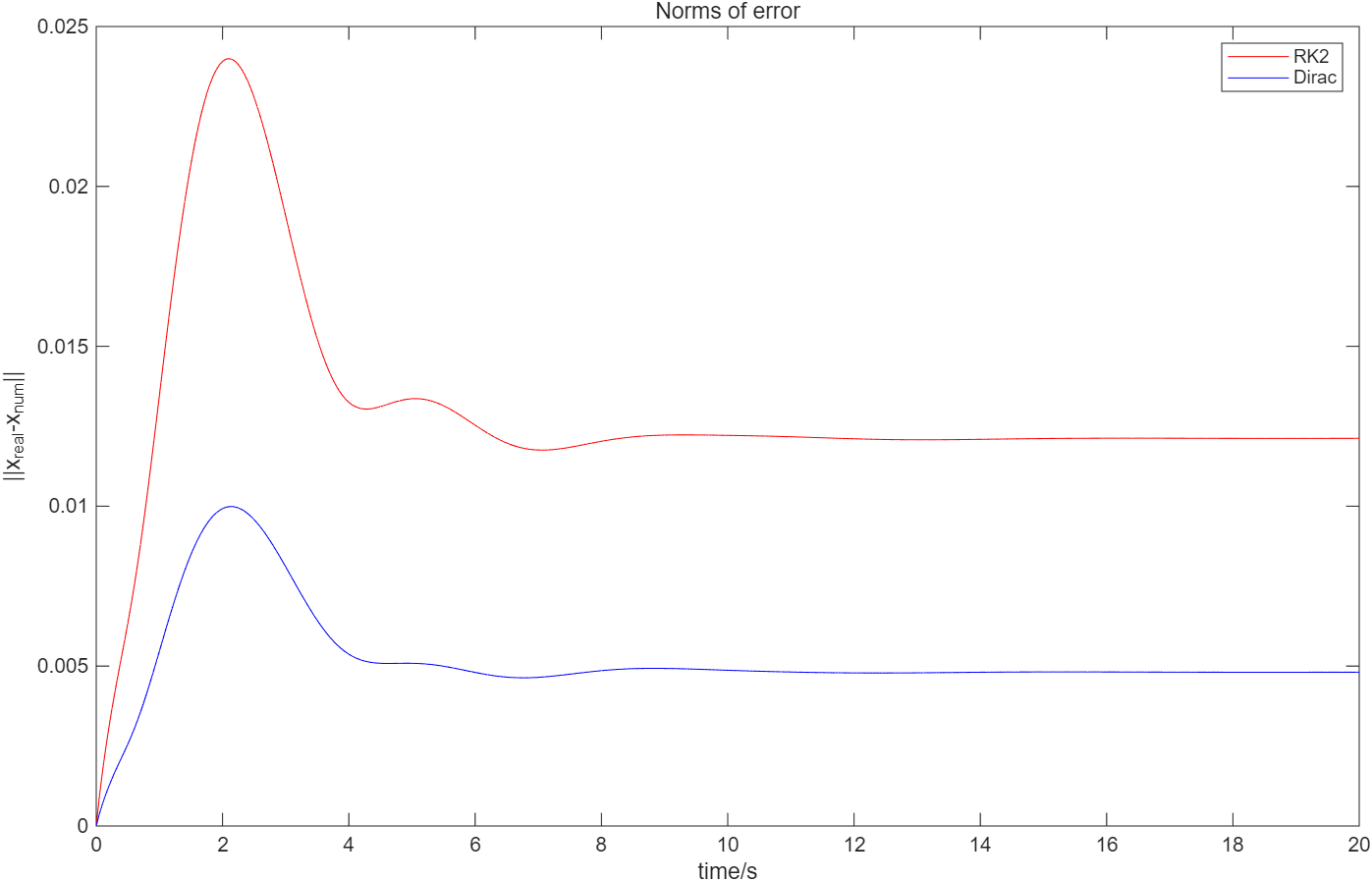}
	\caption{The comparison of two algorithms as $\mathbf{d}=\mathbf{0}$.}
	\label{fig:algorithm}
\end{figure}

\section{Conclusion}\label{conclusion}
Based on the analysis in sections \ref{con} and \ref{numexp}, the proposed controller, grounded in input‑to‑state stability (ISS), is feasible both theoretically and practically. On one hand, it achieves better regulation of $x_1$ through $x_4$ toward zero while using significantly lower control input magnitudes compared to the PI controller. On the other hand, it is not a black‑box controller: the three parameters have clear mathematical interpretations – $\alpha$ governs the convergence rate, $\beta/\alpha$ determines the size of the region to which $H_0$ eventually shrinks, and the third parameter represents the upper bound of the control input.

From sections \ref{algorithm} and \ref{numexp}, the Dirac‑structure‑preserving numerical algorithm exactly conserves the energy $H$ when $\mathbf{d}=\mathbf{0}$. Consequently, the numerical results do not drift significantly, even under strong control inputs. Moreover, by adhering to the energy conservation constraint, the algorithm maintains better control over errors in long‑term simulations.

The most difficult mathematical difficulties overcome in this work are the following. (i) The impossibility of exact disturbance cancellation forced us to adopt an ISS framework rather than standard matching-based energy shaping. This required a careful derivation of the inequality $dH/dt\le -\alpha H+\beta ||\mathbf{d}||^2$ and the construction of a feedback controller that achieves this bound without requiring $B(\mathbf{x})=C(\mathbf{x})K(\mathbf{x})$. The resulting controller involves rational terms $x_i/(x_1^2+x_2^2)$, which are well-defined only away from the origin; we provided a saturation mechanism to handle near-zero denominators. (ii) Preserving the Dirac structure under discretization required an energy-conjugate approximation $\overline{\nabla H}$ such that $H(\mathbf{x}_{k+1})-H(\mathbf{x}_k)=\overline{\nabla H}^\top(\mathbf{x}_{k+1}-\mathbf{x}_k)$. The midpoint rule achieves this with second-order accuracy because $C(\mathbf{x})^\top \nabla H(\mathbf{x})=0$ in our system. For more general port-Hamiltonian systems with nonzero $C^\top \nabla H$, a more complex discrete gradient would be needed. Thus, the specific algebraic structure of the SVG model was exploited to design a simple yet exactly energy-conserving scheme.

Still, two problems remain to be addressed in our future research. First, the parameter representing the upper bound of the control input $\mathbf{u}$ still relies on empirical tuning. Second, in the presence of an external disturbance $\mathbf{d}$ that affects the energy $H$, the algorithm is likely unable to preserve $H$ because $\mathbf{d}$ is typically unknown in practice. 
\section{Acknowledgment}
This study is supported by State Grid Corporation of China Headquarters Science and Technology Project ``Research on Nonlinear Control and Stability of All-Power-Electronics Power Systems'' (Grant No. 52170025002L-464-FGS).
\bibliographystyle{plain}
\bibliography{refs}

@article{qaisar2025grid,
	title={Grid-forming converters for renewable generation: A comprehensive review},
	author={Qaisar, Muhammad Waqas and Fang, Jingyang},
	journal={Energies},
	volume={18},
	number={17},
	pages={4565},
	year={2025},
	publisher={MDPI},
    note = {\href{https://doi.org/10.3390/en18174565}{doi:10.3390/en18174565}}
}

@article{blaabjerg2017distributed,
	title={Distributed power-generation systems and protection},
	author={Blaabjerg, Frede and Yang, Yongheng and Yang, Dongsheng and Wang, Xiongfei},
	journal={Proceedings of the IEEE},
	volume={105},
	number={7},
	pages={1311--1331},
	year={2017},
	publisher={IEEE},
    note = {\href{https://doi.org/10.1109/JPROC.2017.2696878}{doi:10.1109/JPROC.2017.2696878}}
}

@inproceedings{teng2024adaptive,
	title={Adaptive control method of grid-forming svg under non-ideal power grid condition},
	author={Teng, Yufei and Wang, Xi and Shi, Peng and Chang, Zhengwei and Wang, Biao and Ye, Xi and Bai, Jiayu and Han, Lianshan and Dai, Tong and Wang, Yuhong and others},
	booktitle={2024 IEEE 8th Conference on Energy Internet and Energy System Integration (EI2)},
	pages={3237--3241},
	year={2024},
	organization={IEEE},
    note = {\href{https://doi.org/10.1109/EI264398.2024.10991376}{doi:10.1109/EI264398.2024.10991376}}
}

@article{xuadaptive,
	title={Adaptive PI Control of STATCOM for Voltage Regulation},
	author={Xu, Yao and Li, Fangxing},
    journal={IEEE Transactions on Power Delivery},
    year={2014},
    volume={29},
    number={3},
    pages={1002-1011},
    note = {\href{https://doi.org/10.1109/TPWRD.2013.2291576}{doi:10.1109/TPWRD.2013.2291576}}
}

@article{ortega2002interconnection,
	title={Interconnection and damping assignment passivity-based control of port-controlled Hamiltonian systems},
	author={Ortega, Romeo and Van Der Schaft, Arjan and Maschke, Bernhard and Escobar, Gerardo},
	journal={Automatica},
	volume={38},
	number={4},
	pages={585--596},
	year={2002},
	publisher={Elsevier},
    note = {\href{https://doi.org/10.1016/S0005-1098(01)00278-3}{doi:10.1016/S0005-1098(01)00278-3}}
}

@article{van2020port,
	title={Port-Hamiltonian modeling for control},
	author={Van Der Schaft, Arjan},
	journal={Annual Review of Control, Robotics, and Autonomous Systems},
	volume={3},
	number={1},
	pages={393--416},
	year={2020},
	publisher={Annual Reviews},
    note = {\href{https://doi.org/10.1146/annurev-control-081219-092250}{doi:10.1146/annurev-control-081219-092250}}
}

@incollection{maschke1993port,
	title={Port-controlled Hamiltonian systems: modelling origins and systemtheoretic properties},
	author={Maschke, Bernhard M and van der Schaft, Arjan J},
	booktitle={Nonlinear control systems design 1992},
	pages={359--365},
	year={1993},
	publisher={Elsevier},
    note = {\href{https://doi.org/10.1016/B978-0-08-041901-5.50064-6}{doi:10.1016/B978-0-08-041901-5.50064-6}}
}

@article{kumar2025port,
	title={Port-Hamiltonian discontinuous Galerkin finite element methods},
	author={Kumar, Nishant and van der Vegt, Jaap JW and Zwart, Hans J},
	journal={IMA Journal of Numerical Analysis},
	volume={45},
	number={1},
	pages={354--403},
	year={2025},
	publisher={Oxford University Press},
    note = {\href{https://doi.org/10.1093/imanum/drae008}{doi:10.1093/imanum/drae008}}
}

@Book{hairer2006geometric,
  author    = {Hairer, Ernst and Lubich, Christian and Wanner, Gerhard},
  publisher = {Springer},
  title     = {Geometric Numerical Integration: Structure-Preserving Algorithms for Ordinary Differential Equations},
  year      = {2006},
  address   = {Berlin},
  edition   = {second},
  isbn      = {978-3-540-30663-4},
  series    = {Springer Series in Computational Mathematics},
  volume    = {31},
}

@article{kong2023control,
	title={Control design of passive grid-forming inverters in port-Hamiltonian framework},
	author={Kong, Le and Xue, Yaosuo and Qiao, Liang and Wang, Fei},
	journal={IEEE Transactions on Power Electronics},
	volume={39},
	number={1},
	pages={332--345},
	year={2023},
	publisher={IEEE},
    note = {\href{https://doi.org/10.1109/TPEL.2023.3319966}{doi:10.1109/TPEL.2023.3319966}}
}

@article{reich1996enhancing,
	title={Enhancing energy conserving methods},
	author={Reich, Sebastian},
	journal={BIT Numerical Mathematics},
	volume={36},
	number={1},
	pages={122--134},
	year={1996},
	publisher={Springer},
    note = {\href{https://doi.org/10.1007/BF01740549}{doi:10.1007/BF01740549}}
}

@article{sontag1989smooth,
	title={Smooth stabilization implies coprime factorization},
	author={Sontag, Eduardo D},
	journal={IEEE transactions on automatic control},
	volume={34},
	number={4},
	pages={435--443},
	year={1989},
    note = {\href{https://doi.org/10.1109/9.28018}{doi:10.1109/9.28018}}
}

@incollection{sontag1995characterizations,
	title={On characterizations of input-to-state stability with respect to compact sets},
	author={Sontag, Eduardo D and Wang, Yuan},
	booktitle={Nonlinear control systems design 1995},
	pages={203--208},
	year={1995},
	publisher={Elsevier},
    note = {\href{https://doi.org/10.1016/B978-0-08-042371-5.50039-3}{doi:10.1016/B978-0-08-042371-5.50039-3}}
}

@book{van2000l2,
	title={L2-gain and passivity techniques in nonlinear control},
	author={Van der Schaft, Arjan},
	year={2000},
	publisher={Springer}
}

@article{hairer2011geometric,
	title={Geometric numerical integration},
	author={Hairer, Ernst and Hochbruck, Marlis and Iserles, Arieh and Lubich, Christian},
	journal={Oberwolfach Reports},
	volume={8},
	number={1},
	pages={825--900},
	year={2011},
    note = {\href{https://doi.org/10.4171/OWR/2006/14}{doi:10.4171/OWR/2006/14}}
}

@book{arnold1992ordinary,
  title={Ordinary differential equations},
  author={Arnold, Vladimir I},
  year={1992},
  publisher={Springer Science \& Business Media}
}

@inproceedings{feng1984difference,
  title={On difference schemes and symplectic geometry},
  author={Feng, Kang},
  booktitle={Proceedings of the 5th international symposium on differential geometry and differential equations},
  year={1984}
}

@inproceedings{mehrmann2019structure,
  title={Structure-preserving discretization for port-Hamiltonian descriptor systems},
  author={Mehrmann, Volker and Morandin, Riccardo},
  booktitle={2019 IEEE 58th Conference on Decision and Control (CDC)},
  pages={6863--6868},
  year={2019},
  organization={IEEE},
  note = {\href{https://doi.org/10.1109/cdc40024.2019.9030180}{doi:10.1109/cdc40024.2019.9030180}}
}

@Book{tang2026geometric,
  author    = {Yifa Tang and Ruili Zhang and Beibei Zhu and Aiqing Zhu and Sixu Wu},
  publisher = {Science Press},
  title     = {Geometric Algorithms for Dynamical Systems with Applications (in Chinese)},
  year      = {2026},
  address   = {Beijing},
  isbn      = {9787030846334},
  month     = mar,
  series    = {Modern Mathematical Basic Series},
  volume    = {215},
  language  = {Chinese},
  pages     = {629},
}

@Book{Feng2020,
  author    = {Kang Feng},
  publisher = {Science Press},
  title     = {Collected Works of Feng Kang},
  year      = {2020},
  address   = {Beijing},
  edition   = {second},
  isbn      = {9787030656209},
  month     = jul,
}
\end{document}